\documentclass[12pt, reqno]{amsart}
\usepackage{amsmath,amssymb,amsfonts,amsthm}
\usepackage{color}

\usepackage[colorlinks=true,linkcolor=blue,citecolor=blue,urlcolor=blue,pdfborder={0 0 0}]{hyperref}
\usepackage{cleveref}

\usepackage{caption}
\usepackage{thmtools}

\usepackage{mathrsfs}

\crefname{theorem}{Theorem}{Theorems}
\crefname{thm}{Theorem}{Theorems}
\crefname{lemma}{Lemma}{Lemmas}
\crefname{lem}{Lemma}{Lemmas}
\crefname{remark}{Remark}{Remarks}
\crefname{prop}{Proposition}{Propositions}
\crefname{defn}{Definition}{Definitions}
\crefname{corollary}{Corollary}{Corollaries}
\crefname{conjecture}{Conjecture}{Conjectures}
\crefname{question}{Question}{Questions}
\crefname{chapter}{Chapter}{Chapters}
\crefname{section}{Section}{Sections}
\crefname{figure}{Figure}{Figures}
\newcommand{\crefrangeconjunction}{--}

\theoremstyle{plain}
\newtheorem{thm}{Theorem}
\newtheorem*{thmb}{Theorem}

\newtheorem{lem}[thm]{Lemma}
\newtheorem{corollary}[thm]{Corollary}

\newtheorem*{conjectureb}{Conjecture}

\theoremstyle{definition}

\theoremstyle{remark}

\newtheorem*{remarkb}{Remark}

\numberwithin{equation}{section}

\renewcommand{\P}{\mathbb P}

\newcommand{\Z}{\mathbb Z}

\usepackage[margin=1in]{geometry}
\usepackage{extarrows}
\usepackage{mathtools}
\newcommand{\eps}{\varepsilon}

\author{Tom Hutchcroft}
\address{University of British Columbia}
\email{thutch@math.ubc.ca}

\title[Percolation on Quasi-Transitive Graphs of Exponential Growth]{Critical Percolation on Any Quasi-Transitive Graph of Exponential Growth Has No Infinite Clusters}

\begin{document}
\maketitle

\begin{abstract} We prove that critical percolation on any quasi-transitive graph of exponential volume growth does not have a unique infinite cluster. This allows us to deduce from earlier results that critical percolation on any graph in this class does not have any infinite clusters. The result is new when the graph in question is either amenable or nonunimodular. 
\end{abstract}

\section{Introduction}

In \textbf{Bernoulli bond percolation}, each edge of a graph $G=(V,E)$ (which we will always assume to be connected and locally finite) is either deleted or retained at random  with retention probability $p\in[0,1]$, independently of all other edges. We denote the random graph obtained this way by $G[p]$. Connected components of $G[p]$ are referred to as \textbf{clusters}. 
Given a graph $G$, the \textbf{critical probability}, denoted $p_c(G)$ or simply $p_c$, is defined to be 
\[p_c(G)= \sup \left\{p\in [0,1] : G[p] \text{ has no infinite clusters almost surely}\right\}. \]
 A central question concerns the existence or non-existence of infinite clusters at the critical probability $p=p_c$. Indeed, proving 
 that critical percolation on the hypercubic lattice $\Z^d$ has no infinite clusters for every
 $d\geq2$ is perhaps the best known open problem in modern probability theory. Russo \cite{russo1981critical} proved that critical percolation on the square lattice $\Z^2$ has no infinite clusters, while Hara and Slade \cite{hara1994mean}  proved that critical percolation on $\Z^d$  has no infinite clusters for all $d\geq 19$. More recently, Fitzner and van der Hofstad \cite{fitzner2015nearest} improved upon the Hara-Slade method, proving that critical percolation on $\Z^d$ has no infinite clusters for every $d\geq 11$.

In their highly influential paper \cite{bperc96}, Benjamini and Schramm proposed a systematic study of percolation on general \textbf{quasi-transitive} graphs; that is,  graphs $G=(V,E)$ such that the action of the automorphism group Aut$(G)$ on $V$ has only finitely many orbits (see e.g.\ \cite{LP:book} further background).  They made the following conjecture.

\begin{conjectureb}[Benjamini and Schramm]
Let $G$ be a quasi-transitive graph. If $p_c(G)<1$, then $G[p_c]$ has no infinite clusters almost surely. 
\end{conjectureb}

Benjamini, Lyons, Peres, and Schramm \cite{benjamini1999percolation,benjamini1999critical} verified the conjecture for \emph{nonamenable}, \emph{unimodular}, quasi-transitive graphs, while partial progress has been made for \emph{nonunimodular}, quasi-transitive graphs (which are always nonamenable \cite[Exercise 8.30]{LP:book}) by Tim{\'a}r 
 \cite{timar2006percolation} and by Peres, Pete, and Scolnicov \cite{peres2006critical}.
In this note, we verify the conjecture for all quasi-transitive graphs of exponential growth.

\begin{thm}\label{thm:main}
Let $G$ be a quasi-transitive graph with exponential growth. Then $G[p_c]$ has no infinite clusters almost surely. 
\end{thm}

A corollary of \cref{thm:main} is that $p_c<1$ for all quasi-transitive graphs of exponential growth, a result originally due to Lyons \cite{lyons1995random}.

We prove \cref{thm:main} by combining 
 the works of Benjamini, Lyons, Peres, and Schramm~\cite{benjamini1999critical} and Tim{\'a}r \cite{timar2006percolation} with the following simple connectivity decay estimate.
 Given a graph $G$, we write $B(x,r)$ to denote the graph distance ball of radius $r$ about a vertex $x$ of $G$. Recall that a graph $G$ is said to have \textbf{exponential growth} if
\[\mathrm{gr}(G) := \liminf_{r\to\infty} |B(x,r)|^{1/r}\]
is strictly greater than $1$ whenever $x$ is a vertex of $G$. It is easily seen that $\mathrm{gr}(G)$ does not depend on the choice of $x$. 
Let $\tau_p(x,y)$ be the probability that $x$ and $y$ are connected in $G[p]$, and let
$\kappa_{p}(n):= \inf\left\{\tau_{p}(x,y) : x, y \in V,\, d(x,y) \leq n \right\}$.


\begin{thm}\label{thm:pcconnectivitydecay}
Let $G$ be a quasi-transitive graph with exponential growth. Then
\[\kappa_{p_c}(n):= \inf\left\{\tau_{p_c}(x,y) : x, y \in V,\, d(x,y) \leq n \right\}\leq \mathrm{gr}(G)^{-n}\]
for all $n\geq1$.
\end{thm}

\begin{remarkb}
The upper bound on $\kappa_{p_c}(n)$ in \cref{thm:pcconnectivitydecay} is attained when $G$ is a regular tree. 
\end{remarkb}

Following the work of Lyons, Peres, and Schramm \cite[Theorem 1.1]{LPS06}, \cref{thm:main} has the following immediate corollary, which is new in the amenable case. The reader is referred to \cite{LPS06} and \cite{LP:book} for background on minimal spanning forests.

\begin{corollary}
Let $G$ be a unimodular quasi-transitive graph of exponential growth. Then every component of the wired minimal spanning forest of $G$ is one-ended almost surely.
\end{corollary}

\section{Proof}

\begin{proof}[Proof of \cref{thm:main} given \cref{thm:pcconnectivitydecay}]
Let us recall the following results: 
\begin{thmb}[Newman and Schulman {\cite{newman1981infinite}}]
Let $G$ be a quasi-transitive graph. Then $G[p]$  has either no infinite clusters, a unique infinite cluster, or infinitely many infinite clusters almost surely for every $p \in [0,1]$. 
\end{thmb}
\begin{thmb}[Burton and Keane {\cite{burton1989density}}, Gandolfi, Keane, and Newman {\cite{gandolfi1992uniqueness}}]
Let $G$ be an amenable quasi-transitive graph. Then $G[p]$ has at most one infinite cluster almost surely for every $p\in[0,1]$. 
\end{thmb}

\begin{thmb}[Benjamini, Lyons, Peres, and Schramm {\cite{benjamini1999percolation,benjamini1999critical}}]
Let $G$ be a nonamenable, unimodular, quasi-transitive graph. Then $G[p_c]$ has no infinite clusters almost surely.
\end{thmb}
\begin{thmb}[Tim{\'a}r {\cite{timar2006percolation}}]
 Let $G$ be a nonunimodular, quasi-transitive graph. Then $G[p_c]$ has at most one infinite cluster almost surely.
\end{thmb}
The statements given for the first two theorems above are not those given in the original papers; 
 the reader is referred to \cite{LP:book} for a modern account of these theorems and for the definitions of unimodularity and amenability.  For our purposes, the significance of the above theorems is that, to prove \cref{thm:main}, it suffices to prove that if $G$ is a quasi-transitive graph of exponential growth, then $G[p_c]$ does not have a unique infinite cluster almost surely.
This follows immediately from \cref{thm:pcconnectivitydecay}, 
 since if $G[p_c]$ contains a unique infinite cluster then
\begin{align*}\tau_{p_c}(x,y) &\geq \P_{p_c}(x \text{ and } y \text{ are both in the unique infinite cluster}) \\
&\geq \P_{p_c}(x \text{ is in the unique infinite cluster})\P_{p_c}(y \text{ is in the unique infinite cluster}) \end{align*}
for all $x,y \in V$ by Harris's inequality \cite{harris1960lower}. Quasi-transitivity implies that the right hand side is bounded away from zero if $G[p_c]$ contains a unique infinite cluster almost surely, and consequently that $\lim_{n\to\infty}\kappa_{p_c}(n)>0$ in this case. 
\end{proof}

\begin{lem}\label{lem:superkappa} Let $G$ be any graph. Then $\kappa_p(n)$ is a supermultiplicative function of $n$. That is, for every $p,n$ and $m$, we have that $\kappa_p(m +n) \geq \kappa_p(m)\kappa_p(n)$.
\end{lem}
\begin{proof}
Let $u$ and $v$ be two vertices with $d(u,v)\leq m+n$. Then there exists a vertex $w$ such that $d(u,w)\leq m$ and $d(w,v)\leq n$. Since the events $\{u \leftrightarrow w\}$ and $\{w \leftrightarrow v\}$  are increasing and 
$\{u\xleftrightarrow{} v\} \supseteq \{u \leftrightarrow w\} \cap \{w \leftrightarrow v\}$, Harris's inequality \cite{harris1960lower} implies that
\[\tau_p(u,v) \geq \tau_p(u,w)\tau_p(w,v) \geq \kappa_p(m)\kappa_p(n).\]
The claim follows by taking the infimum.   
\end{proof}

\begin{lem}\label{lem:leftcont} Let $G$ be a quasi-transitive graph. Then $\sup_{n\geq1} (\kappa_p(n))^{1/n}$ is left continuous in $p$. That is, 
\begin{equation}\label{eq:kappar} \adjustlimits\lim_{\eps\to0+}\sup_{n\geq1} \left(\kappa_{p-\eps}(n)\right)^{1/n}=\sup_{n\geq 1} \left(\kappa_p(n)\right)^{1/n} 
\quad \text{ for every $p\in(0,1]$.}
\end{equation}
\end{lem}
\begin{proof}
Recall that an increasing function is left continuous if and only if it is lower semi-continuous, and that lower semi-continuity is preserved by taking minima (of finitely many functions) and suprema (of arbitrary collections of functions). 
Now, observe that $\tau_p(x,y)$ is lower semi-continuous in $p$ for each pair of fixed vertices $x$ and $y$: 
 This follows from the fact that $\tau_p(x,y)$ can be written as the supremum of the continuous functions $\tau_p^r(x,y)$, which give the probabilities that $x$ and $y$ are connected in $G[p]$ by a path of length at most $r$.
Since $G$ is quasi-transitive, there are only finitely many isomorphism classes of pairs of vertices at distance at most $n$ in $G$, and we deduce that $\kappa_p(n)$ is also lower semi-continuous in $p$ for each fixed $n$. 
Thus, $\sup_{n\geq1} (\kappa_p(n))^{1/n}$ is a supremum of lower semi-continuous functions and is therefore lower semi-continuous itself. 
\end{proof}

We will require the following well known theorem.

\begin{thm}\label{thm:susceptibility}
Let $G$ be a quasi-transitive graph, and let $\rho$ be a fixed vertex of $G$. Then the expected cluster size is finite for every $p<p_c$. That is,
\[\sum_x\tau_p(\rho,x)<\infty \quad \text{ for every $p<p_c$.}\]
\end{thm}

 This theorem was proven in the transitive case by Aizenmann and Barsky~\cite{aizenman1987sharpness}, and in the quasi-transitive case by Antunovi{\'c} and Veseli{\'c} \cite{antunovic2008sharpness}; see also the recent work of Duminil-Copin and Tassion \cite{duminil2015new} for a beautiful new proof in the transitive case.

\begin{proof}[Proof of \cref{thm:pcconnectivitydecay}]
Let $\rho$ be a fixed root vertex of $G$. For every $p\in[0,1]$ and every $n\geq1$, we have
\begin{align*}
\kappa_p(n)\cdot \left|B(\rho,n)\right| \leq \sum_{x\in B(\rho,n)} \tau_{p}(\rho,x) \leq \sum_x\tau_{p}(\rho,x). \end{align*}
Thus, it follows from \cref{thm:susceptibility}, \cref{lem:superkappa}, and
Fekete's Lemma that
\[\sup_{n\geq1} (\kappa_p(n))^{1/n} = \lim_{n\to\infty} (\kappa_p(n))^{1/n} \leq \limsup_{n\to\infty} \left(\frac{\sum_x\tau_{p}(\rho,x)}{|B(\rho,n)|}\right)^{1/n} = \mathrm{gr}(G)^{-1}\] for every $p<p_c$. We conclude by applying \cref{lem:leftcont}.
\end{proof}

\subsection*{Acknowledgments}
This work was supported by a Microsoft Research PhD Fellowship and was carried out while the author was an intern at Microsoft Research, Redmond. We thank Ander Holroyd, Gady Kozma, Russ Lyons, and Asaf Nachmias for helpful comments and suggestions. 

\footnotesize{
  \bibliographystyle{abbrv}
  \bibliography{unimodular}
  }

\end{document}